\documentclass[a4paper,draft]{amsproc}
\usepackage{amssymb}
\usepackage{amscd} 
\theoremstyle{plain}
\newtheorem*{theoA}{Theorem A}
\newtheorem*{theoB}{Theorem B}
\newtheorem*{theoC}{Theorem C}
\newtheorem*{theoD}{Theorem D}
\newtheorem*{theoE}{Theorem E}
\newtheorem*{theoF}{Theorem F}
\newtheorem*{theoG}{Theorem G}

\newtheorem{theo}{Theorem}[section]

\newtheorem{lem}{Lemma}[section]

\theoremstyle{definition}

\newtheorem{exm}{Example}[section]
\newtheorem{ques}{Question}[section]
\newtheorem{defi}{Definition}[section]

\theoremstyle{remark}
\newtheorem{rem}{Remark}[section]

\newcommand{\ol}{\overline}
\newcommand{\be}{\begin{equation}}
\newcommand{\ee}{\end{equation}}
\newcommand{\beas}{\begin{eqnarray*}}
	\newcommand{\eeas}{\end{eqnarray*}}
\newcommand{\bea}{\begin{eqnarray}}
\newcommand{\eea}{\end{eqnarray}}

\numberwithin{equation}{section}

\renewcommand{\leq}{\leqslant}
\renewcommand{\geq}{\geqslant}
\renewcommand{\setminus}{\smallsetminus}
\title[Class of meromorphic functions partially shared values ...]{\LARGE
	Class of meromorphic functions partially shared values with their
	differences or shifts}

\numberwithin {equation}{section}
\date{}
\author{Molla Basir Ahamed}

\address{ Molla Basir Ahamed, Department of Mathematics, Kalipada Ghosh Tarai Mahavidyalya, West Bengal, 734014, India.}
\email{basir\_math\_kgtm@yahoo.com, bsrhmd117@gmail.com.}
\address{ Present Address: Molla Basir Ahamed, Department of Mathematics, Jadavpur University, Kolkata-700032, West Bengal, India.}
\email{mbahamed.math@jadavpuruniversity.in}
\begin{document}
	\vspace{18mm} \setcounter{page}{1} \thispagestyle{empty}
	
\begin{abstract}
Two meromorphic functions $ f $ and $ g $ are said to share a value $ s\in\mathbb{C}\cup\{\infty\} $ $ CM $ $ (IM) $ provided that $ f(z)-s $ and $ g(z)-s $ have the same set of zeros counting multiplicities (ignoring multiplicities). We say that a meromorphic function $ f $ share $ s\in\mathbb{C} $
partially $ CM $ with a meromorphic function $ g $ if $ E(s,f)\subseteq  E(s,g)$. It is easy to see that the condition ``partially shared values $ CM $" is more general than the condition ``shared value $ CM $". With the idea of partially shared values, in this paper, we prove some uniqueness results between non-constant meromorphic functions and their shifts or generalized differences. We exhibit some examples to show that the result of {Charak \emph{et al.}} \cite{Cha & Kor & Kum-2016} is not true for $k=2 $ or $ k=3 $. We find some gaps in proof of the result of {Lin} \emph{et al.} \cite{Lin & Lin & Wu}, and we not only correct them but also generalize their result in a more convenient way. A number of examples have been exhibited to validate certain claim of the main results of this paper and also to show that some of the conditions are sharp. In the end, we have posed some open questions for further investigation of the main result of the paper.   	
\end{abstract}
	\maketitle
	
\section{Introduction}
	We assume that the reader is familiar with the elementary Nevanlinna theory,	for detailed information, we refer the reader \cite{Goldberg,Hay-1964,Laine-1993}. Meromorphic functions considered in this paper are always non-constant, unless otherwise specified.
	For such a function $f$ and $a\in\mathbb{\ol C}=:\mathbb{C}\cup\{\infty\}$, each  $z$ with $f(z)=a$ will be called $a$-point of $f$.  We will use  here some standard definitions and basic notations from this theory. In particular by $N(r,a;f)$ ($\ol N(r,a;f)$), we denote the counting function (reduced counting function) of $a$-points of meromorphic functions $f$, $T(r,f)$ is the Nevanlinna	characteristic function of $f$ and $S(r,f)$ is used to denote each functions which is of smaller order than $T(r,f)$ when $r\rightarrow \infty$.  We also denote $ \mathbb{C^{*}} $ by $\mathbb{C^{*}}:=\mathbb{C}\setminus\{0\}$. \vspace{1mm}
	
	\par For a meromorphic function $ f $, the order $ \rho(f) $
	and the hyper order $ \rho_2(f) $ of $ f $ are defined respectively by
	\beas
	\rho(f)=\limsup_{r\rightarrow\infty}\frac{\log^{+}T(r,f)}{\log
		r}\;\;\mbox{and}\;\;\rho_2(f)=\limsup_{r\rightarrow\infty}\frac{\log^{+}\log^{+}T(r,f)}{\log
		r}. \eeas
	\par For $ a\in\mathbb{C}\cup\{\infty\} $, we also define \beas
	\Theta(a;f)=1-\limsup_{r\rightarrow +\infty}\frac{\ol
		N\left(r,{1}/{(f-a)}\right)}{T(r,f)}.\eeas
	\par We denote $ \mathcal{S}(f) $ as the family of all meromorphic functions $ s $	for which $ T(r,s)=o(T(r,f)) $, where $ r\rightarrow\infty $ outside of a possible exceptional set of finite logarithmic measure. Moreover, we also include all constant functions in $ \mathcal{S}(f) $, and let $\hat{\mathcal{S}}(f)=\mathcal{S}(f)\cup\{\infty\} $. For $
	s\in\hat{\mathcal{S}}(f) $, we say that two meromorphic functions $ f $ and $ g$ share $ s $ $ CM $ when $ f(z)-s $ and $ g(z)-s $ have the same zeros with the same multiplicities. If multiplicities are not taking into account, then we say that $ f $ and $ g $ share $ s $ $ IM $.\vspace{1mm}
	\par In addition, we denote $ \ol E(s,f) $ by the set of zero of $ f-s $, where a zero	is counted only once in the set, and by the set $\ol  E_{k)}(s,f) $, we understand a set of zeros of $ f-s $ with multiplicity $ p \leq k $, where a zero with
	multiplicity $ p $ is counted only once in the set. Similarly, we denote the reduced counting function corresponding to $ \ol E_{k)}(s,f) $ as $
	\ol N_{k)}\left(r,1/(f-s)\right) $.\vspace{1mm}
	
	\par In the uniqueness theory of meromorphic functions, the the famous classical results  are the five-point, resp. four-point, uniqueness theorems due to Nevanlinna \cite{Nevanlinna-1929}. The five-point theorems states that if two meromorphic functions $f$, $g$
	share five distinct values in the extended complex plane $IM$, then $f\equiv g$. The beauty of this
	result lies in the fact that there is no counterpart of this result in case of real	valued functions. On the other hand, four-point theorem states that if two meromorphic functions $f,\;g$ share four
	distinct values in the extended complex plane $CM$, then
	$f\equiv T\circ g$, where $T$ is a M$\ddot{o}$bius transformation.\vspace{1mm}
	
	\par Clearly, these	results initiated the study of uniqueness of two meromorphic functions $f$ and $g$. The study of such uniqueness theory becomes more interesting if the function $g$ has some expressions in terms of	$f$.\vspace{1mm}
	
	\par Next we explain the following definition which will be required in the sequel.
	\begin{defi}
		Let $f$ and $g$ be two meromorphic functions such that $f$ and $g$ share the
		value $a$ with weight $k$ where $a\in\mathbb{C}\cup\{\infty\}$. We denote by
		$\ol N_E^{(k+1}\left(r,1/(f-a)\right)$ the counting function
		of those $a$-points of $f$ and $g$ where $p=q\geq k+1$, each point in this
		counting function counted only once.  
	\end{defi}\par 
	In what follows, let $c$ be a non-zero constant. For a meromorphic
	function $f$, let us denote its shift $I_{c}f$ and difference operators
	$\Delta_{c}f$,  respectively, by $I_{c}f(z)=f(z+c)$ and
	$\Delta_{c}f(z)=(I_{c}-1)f(z)=f(z+c)-f(z).$\vspace{1mm} 
	
	\par Recently an increasing amount of
	interests have been found among the researchers to find results which are the difference analogue of Nevanlinna theory. For finite ordered meromorphic	functions, {Halburd and Korhonen} \cite{Hal & Kor-JMMA-2006},  and  {Chiang and Feng} \cite{Chi & Fen-2008} developed independently parallel difference version of the famous Nevanlinna theory. As applications of this theory, we refer the reader to see the articles in case of set sharing problems (see, for example \cite{Ahamed-SUBB-2019,Aha-TJA-2021,Ban-Aha-Filomat-2019,Ban-Aha-MS-2020,Che-Che-BMMSS-2012,Zha-JMMA-2010}), finding solutions to the Fermat-type difference equations (see e.g. \cite{Aha-JCMA-2021,Liu-AM-2012,Cao-MJM-2018}), Nevanlinna theory of the Askey–Wilson divided difference operators (see e.g. \cite{Chiang-Feng-AM-2010}), meromorphic solutions to the difference equations of Malmquist type (see e.g. \cite{Lu-BAMS-2016}) and references therein.\vspace{1mm}
	
	\par Regarding periodicity of meromorphic functions, {Heittokangas} et. al. \cite{Hei & Kor & Lai & Rie-CVTA-2001,Hei & Kor &
		 Lai & Rie-JMMA-2009} have considered the problem of value sharing for shifts of meromorphic functions and obtained the following result.
	\begin{theoA}\cite{Hei & Kor & Lai & Rie-CVTA-2001}
		Let $ f $ be a meromorphic function of finite order, and let $
		c\in\mathbb{C^{*}} $. If $ f(z) $ and $ f(z+c) $ share three distinct periodic
		functions $ s_1, s_2, s_3\in\hat{\mathcal{S}}(f) $ with period $ c $ $ CM $,
		then $ f(z)\equiv f(z+c) $ for all $ z\in\mathbb{C} $.
	\end{theoA}
	\par In $ 2009 $, Heittokangas \emph{et al.} \cite{Hei & Kor & Lai & Rie-JMMA-2009}  improved {Theorem A} by replacing ``sharing three small functions $ CM $" by ``$ 2\; CM + 1\; IM $" and obtained the following result.
	\begin{theoB}\cite{Hei & Kor & Lai & Rie-JMMA-2009}
		Let $ f $ be a meromorphic function of finite order, and let $
		c\in\mathbb{C^{*}} $. Let $ s_1, s_2, s_3\in\hat{\mathcal{S}}(f) $ be three
		distict periodic function with period $ c $. If $ f(z) $ and $ f(z+c) $ share  $
		s_1, s_2\in\hat{\mathcal{S}}(f) $  $ CM $ and $ s_3 $ $ IM $, then $ f(z)\equiv
		f(z+c) $ for all $ z\in\mathbb{C} $.
	\end{theoB}\par 
	
In $ 2014 $, {Halburd} \emph{et al.} \cite{Hal & Kor &
Toh-TAMS-2014} extended some results in this direction to meromorphic functions $ f $  whose hyper-order $ \rho_2(f)$ less than one. One may get much more information from \cite{Aha-JCMA-2021,Aha-TJA-2021,Che & Lin-2016,Hei & Kor & Lai & Rie-CVTA-2001,Hei & Kor & Lai &
		Rie-JMMA-2009,Liu-JMMA-2009,Liu & Yan-AM-2009} and the references therein, about the relationship between a
	meromorphic function $ f(z) $ and it shift $ f(z+c) $.\vspace{1mm}
	
	\par In $ 2016 $, {Li and Yi} \cite{Li & Yi-BKMS-2016} obtained a uniqueness	result of meromorphic functions $ f $ sharing four values with their shifts $	f(z+c) $.
	\begin{theoC}\cite{Li & Yi-BKMS-2016}
		Let $ f $ be a non-constant meromorphic function of hyper-order $ \rho_2(f)<1 $
		and $ c\in\mathbb{C^{*}} $. Suppose that $ f $ and $ f(z+c) $ share $ 0$, $1$, 
		$\eta $ $ IM $, and share $ \infty $ $ CM $, where $ \eta $ is a finite value
		such that $ \eta\neq 0, 1 $. Then $ f(z)\equiv f(z+c) $ for all $ z\in\mathbb{C}
		$.
	\end{theoC}
We now recall here the definition of partially shared values by two meromorphic functions $ f $ and $ g $.
\begin{defi}\cite{Chen-CMFT-2018}
	Let $ f $ and $ G $ be non-constant meromorphic functions and $ s\in\mathbb{C}\cup\{\infty\} $. Denote the set of all zeros of $ f-s $ by $ E(s,f) $, where a zero of multiplicity $ m $ is counted $ m $ times. If $ E(s,f)\subset E(s,g) $, then we say that $ f $ and $ g $ partially share the value $ s $ $ CM $. Note that $ E(s,f)=E(s,g) $ is equivalent to $ f $ and $ g $ share the value $ s $ $ CM $. Therefore, it is easy to see that the condition ``partially shared values $ CM $" is more general than the condition ``shared value $ CM $". 
\end{defi}
\par In addition, let $ \ol E(s,f) $ denote the set of zeros of $ f-s $, where a zero is counted only once in the set, and $ \ol E_{k)}(s,f) $ denote the set of zeros of $ f-s $ with multiplicity $ l\leq k $, where a zero with multiplicity $ l $ is counted only once in the set. The reduced counting function corresponding to to $ \ol E_{k)}(s,f) $ are denoted by $ \ol N_{k)}(r,1/(f-s)) $.\vspace{1mm}

\par Charak \emph{et al.} \cite{Cha & Kor & Kum-2016} gave the following definition of partial sharing.
	\begin{defi}\cite{Cha & Kor & Kum-2016}
		We say that a meromorphic function  $ f $ share $ s\in\hat{\mathcal{S}} $
		partially with a meromorphic function $ g $ if $ \ol E(s,f)\subseteq \ol E(s,g)
		$, where $ \ol E(s,f) $ is the set of zeros of $ f(z)-s(z) $, where each zero is counted only once.
	\end{defi}
	\par Let $ f $ and $ g $ be two non-constant meromorphic functions and $
	s(z)\in\hat{\mathcal{S}}(f)\cap\hat{\mathcal{S}}(g) $. We denote by $ \ol
	N_0(r,s;f,g ) $ the counting function of common solutions of $ f(z)-s(z)=0 $ and
	$ g(z)-s(z)=0 $, each counted only once. Put \beas \ol N_{12}(r,s;f,g)=\ol
	N\left(r,\frac{1}{f-s}\right)+\ol N\left(r,\frac{1}{g-s}\right)- 2\ol
	N_{0}(r,s;f,g).  \eeas
	It is easy to see that $ \ol N_{12}(r,s;f,g) $ denoted the counting
	function of distinct solutions of the simultaneous equations $ f(z)-s(z)=0 $ and $ g(z)-s(z)=0 $.\vspace{1mm}
	
	\par In $ 2016 $, \textit{Charak} \emph{et al.} \cite{Cha & Kor & Kum-2016} introduced the above notion of partial sharing of values and applying this notion of sharing, they have obtained the following interesting result.
	\begin{theoD}\cite{Cha & Kor & Kum-2016}
		Let $ f $ be a non-constant meromorphic function of hyper order $ \rho_2(f)<1
		$, and $ c\in\mathbb{C^{*}} $. Let $ s_1, s_2, s_3, s_4\in\hat{\mathcal{S}}(f) $
		be four distinct periodic functions with period $ c $. If $ \delta(s,f)>0 $ for
		some $ s\in\hat{\mathcal{S}}(f) $ and \beas \ol E(s_{j},f)\subseteq\ol E(s_{j},
		f(z+c)), \;\;\; \text{j=1, 2, 3, 4,} \eeas then $ f(z)=f(z+c) $ for all $
		z\in\mathbb{C} $.
	\end{theoD} 
	\par In $ 2018 $, \textit{Lin} \emph{et al.} \cite{Lin & Lin & Wu} investigated further on the result of \textit{Charak} \emph{et al.} \cite{Cha & Kor &
		Kum-2016} replacing the condition ``partially shared value $ \ol
	E(s,f)\subseteq\ol E(s,f(z+c)) $" by the condition ``truncated partially shared
	value $ \ol E_{k)}(s,f)\subseteq\ol E_{k)}(s,f(z+c)) $", $ k $ is a positive integer. By the following example, \textit{Lin} et. al. \cite{Lin & Lin & Wu} have shown that	the result of \emph{Charak} et. al. \cite{Cha & Kor & Kum-2016} is not be true	for $ k=1 $ if truncated partially shared values is considered.
	\begin{exm}\cite{Lin & Lin & Wu}
		Let $ f(z)={2e^z}/{(e^{2z}+1)} $ and $ c=\pi i $, $ s_1=1 $, $
		s_2=-1 $, $ s_3=0 $, $ s_4=\infty $ and $ k=1 $. It is easy to see that $ f(z+\pi i)=-{2e^z}/{(e^{2z}+1)}$ and $ f(z) $ satisfies all the other	conditions of {Theorem D},  but $
		f(z)\not\equiv f(z+c) $.
	\end{exm}\par 
	
However, after a careful investigation, we find that {Theorem D} is not valid in fact for each positive integer $k $ although $ f(z) $ and $ f(z+c) $ share value $ s\in\{s_1, s_2, s_3, s_4\} $ $ CM $. We give here only two examples for $ k=2 $ and $ k=3 $.
	\begin{exm}
		Let $ f(z)={\left(ae^z(e^{2z}+3)\right)}/{\left(3e^{2z}+1\right)} $, $ c=\pi i $ and $
		s_1=a $, $ s_2=-a $, where $ a\in\mathbb{C^{*}} $, $ s_3=0 $, $ s_4=\infty $ and
		$ k_1=2=k_2 $. It i easy to see that $ f(z+\pi
		i)=-{\left(ae^z(e^{2z}+3)\right)}/{\left(3e^{2z}+1\right)} $ and $ f(z) $ satisfies all
		the conditions of {Theorem D}, but $ f(z)\not\equiv f(z+c) $. 
	\end{exm}
	\begin{exm}
		Let $ f(z)={\left(4ae^z(e^{2z}+1)\right)}/{\left(e^{4z}+6e^{2z}+1\right)} $ and $ c=\pi i
		$, $ s_1=a $, $ s_2=-a $, where $ a\in\mathbb{C^{*}} $, $ s_3=0 $, $ s_4=\infty
		$ and $ k_1=3=k_2 $. Then clearly $ f(z+\pi
		i)=-{\left(4ae^z(e^{2z}+1)\right)}/{\left(e^{4z}+6e^{2z}+1\right)} $ and $ f(z) $
		satisfies all the conditions of {Theorem D}, but $ f(z)\not\equiv f(z+c) $. 
	\end{exm} 
	\par In $ 2018 $, \textit{Lin} \emph{et al.} \cite{Lin & Lin & Wu} established the	following result considering partially sharing values.
	\begin{theoE}\cite{Lin & Lin & Wu}
		Let $ f $ be a non-constant meromorphic function of hyper-order $ \rho_2(f)<1 $
		and $ c\in\mathbb{C^{*}} $. Let $ k_1, k_2 $ be two positive integers, and let $
		s_1, s_2\in\mathcal{S}(f)\cup\{0\} $, and $ s_3, s_4\in\hat{\mathcal{S}}(f) $ be
		four distinct periodic functions with period $ c $ such that $ f $ and $ f(z+c)
		$ share $ s_3, s_4 $ $ CM $ and \beas \ol E_{k_{j})}(s_{j},f)\subseteq \ol
		E_{k_{j})}(s_{j},f(z+c)),\;\; j=1, 2.  \eeas If $
		\Theta(0,f)+\Theta(\infty;f)>{2}/{(k+1)} $, where $ k=\min\{k_1,
		k_2\} $, then $ f(z)\equiv f(z+c) $ for all $ z\in\mathbb{C} $.
	\end{theoE} \par
	
	\par As a consequence of {Theorem E}, \textit{Lin} \emph{et al.} \cite{Lin &	Lin & Wu} obtained the following result.
	\begin{theoF} \cite{Lin & Lin & Wu}
		Let $ f $ be a non-constant meromorphic function of hyper order $ \rho_2(f)<1
		$, $ \Theta(\infty,f)=1 $ and $ c\in\mathbb{C^{*}} $. Let $ s_1, s_2,
		s_3\in\mathcal{S}(f) $ be three distinct periodic functions with period $ c $
		such that $ f(z) $ and $ f(z+c) $ share $ s_3 $ $ CM $ and \beas \ol
		E_{k)}(s_j,f)\subseteq \ol E_{k)}(s_j,f(z+c)),\;\; j=1, 2. \eeas If $ k\geq 2 $,
		then $ f(z)\equiv f(z+c) $ for all $ z\in\mathbb{C} $.
	\end{theoF} 
 \textit{Lin} \emph{et al.} \cite{Lin &	Lin & Wu} have showed that number ``$ k= 2 $" is sharp for the function $ f(z)=\sin z $ and $ c=\pi $. It is easy to see that $ f(z+c) $ and $ f(z) $ share the value $ 0 $ $ CM $ and $ \ol E_{1)}(1,f(z))= \ol E_{1)}(1,f(z+c))=\phi $ and $ \ol E_{1)}(-1,f(z))= \ol E_{1)}(-1,f(z+c))=\phi $ but $ f(z+c)\not\equiv f(z) $. Since Theorem F is true for $ k\geq 2 $, hence \textit{Lin} \emph{et al.} \cite{Lin & Lin & Wu} investigated further to explore the situation when $ k=1 $ and obtained the  result. 
	\begin{theoG} \cite{Lin & Lin & Wu}
		Let $ f $ be a non-constant meromorphic function of hyper order $ \rho_2(f)<1
		$, $ \Theta(\infty,f)=1 $ and $ c\in\mathbb{C^{*}} $. Let $ s_1, s_2,
		s_3\in\mathcal{S}(f) $ be three distinct periodic functions with period $ c $
		such that $ f(z) $ and $ f(z+c) $ share $ s_3 $ $ CM $ and \beas \ol
		E_{1)}(s_j,f)\subseteq \ol E_{1)}(s_j,f(z+c)),\;\; j=1, 2. \eeas Then $
		f(z)\equiv f(z+c) $ or $ f(z)\equiv - f(z+c) $ for all $ z\in\mathbb{C} $.
		Moreover, the later occurs only if $ s_1+s_2=2s_3 .$
	\end{theoG}
\begin{rem}
We find in the proof of \cite[Theorem 1.6]{Lin & Lin & Wu}, \textit{Lin} \emph{et al.} made a mistake. In {Theorem 1.6}, they have obtained  $ f(z+c)\equiv -f(z) $ as one of the conclusion under the condition $s_1+s_2=2s_3 $, where correct one it will be $ f(z+c)\equiv -f(z)+2s_3 $. One can easily understand it from the following explanation. In \cite[Proof
of Theorem 1.6, page - 476]{Lin & Lin & Wu} the authors have obtained $ \alpha=-1 $,	where $ \alpha $, the way they have defined, finally will be numerically equal with ${\left(f(z+c)-s_3\right)}/{\left(f(z)-s_3\right)}=\alpha $, when $ s_1+s_2=2s_3 $.	Hence after combining, it is easy to see that ${\left(f(z+c)-s_3\right)}{\left(f(z)-s_3\right)}=-1 $ and this implies that $ f(z+c)\equiv -f(z)+2s_3.$
	\end{rem}
\par In this paper, taking care of these points. our aim is to extend the above results with certain suitable setting. Henceforth, for a meromorphic	function $ f $ and $ c\in\mathbb{C^{*}} $, we recall here (see \cite{Aha & RM & 2019}) $ \mathcal{L}_c(f):	=c_1f(z+c)+c_0f(z) $, where $ c_1 (\neq 0), c_0\in\mathbb{C} $. Clearly, $\mathcal{L}_c(f) $ is a generalization of shift $ f(z+c) $ as well as the difference operator $ \Delta_{c}f $.\vspace{1mm} 
	
	\par To give a correct version of the result of Lin \emph{et al.} with a general setting, we are mainly interested  to find the affirmative answers of the following questions.
	\begin{ques}
		Is it possible to extend $ f(z+c) $ upto $ \mathcal{L}_c(f) $, in all the above mentioned results?
	\end{ques}
	\begin{ques}
		Can we obtained a similar result of Theorem E, replacing the condition $
		\Theta(0;f)+\Theta(\infty;f)>{2}/{(k+1)} $, where $ k=\min\{k_1,
		k_2\} $ by a more general one?
	\end{ques} If the answers of the above question are found to be affirmative, then it is natural to raise the following questions.
	\begin{ques}
		Is the new general condition, so far obtained, sharp?	
	\end{ques}
	\begin{ques}
		Can we find the class of all the meromorphic function which satisfies the difference equation $ \mathcal{L}_c(f)\equiv f $?
	\end{ques}
Answering the above questions is the main objective of this paper. We organize the paper as follows: In section 2, we state the main results of this paper and exhibit several examples pertinent with the different issues regarding the main results. In section 3, key lemmas are stated and some of them are proved. Section 4 is devoted specially to prove the main results of this paper. In section 5, some questions have raised for further investigations on the main results of this paper.
\section{Main Results}
	We prove the following result generalizing that of \textit{Lin} \emph{et al.} \cite{Lin & Lin & Wu}.
		\begin{theo}\label{th2.1}
		Let $ f $ be a non-constant meromorphic function of hyper order $ \rho_2(f)<1	$ and $ c, c_1\in\mathbb{C^{*}} $. Let $ k_1 $, $ k_2 $ be two positive integers, and $ s_1 $, $ s_2\in\mathcal{S}\setminus\{0\} $, $ s_3 $, $
		s_4\in\hat{\mathcal{S}}(f) $ be four distinct periodic functions with period $ c	$ such that $ f $ and $ \mathcal{L}_c(f) $ share $ s_3, $ $ s_4 $ $ CM $ and
		\beas \ol E_{k_{j})}(s_j,f)\subseteq \ol E_{k_{j})}(s_j,\mathcal{L}_c(f)),\;\;
		j=1, 2.  \eeas If \beas
		\Theta(0;f)+\Theta(\infty;f)>\frac{1}{k_1+1}+\frac{1}{k_2+1}, \eeas then $\mathcal{L}_c(f)\equiv f $. Furthermore, $ f $ assumes the following form \beas
		f(z)=\left(\frac{1-c_0}{c_1}\right)^{\displaystyle\frac{z}{c}}g(z) ,\eeas where $ g(z) $ is a meromorphic function such that $ g(z+c)=g(z) $, for all $ z\in\mathbb{C} $.
	\end{theo}

\begin{rem}
	The following examples show that the condition\beas
	\Theta(0;f)+\Theta(\infty;f)>\frac{1}{k_1+1}+\frac{1}{k_2+1}\eeas  in
	{Theorem \ref{th2.1}}  is sharp.
\end{rem}
	
	\begin{exm}
		Let $ f(z)={\left(ae^z(e^{2z}+3)\right)}/{\left(3e^{2z}+1\right)} $, $ c=\pi i $ and $
		s_1=a $, $ s_2=-a $, where $ a\in\mathbb{C^{*}} $, $ s_3=0 $, $ s_4=\infty $ and
		$ k_1=2=k_2 $. It is easy to see that $ \mathcal{L}_{\pi
			i}(f)=-{\left(ae^z(e^{2z}+3)\right)}/{\left(3e^{2z}+1\right)} $, where $ c_1=c_0+1 $, $ c_0, c_1\in\mathbb{C^{*}} $, and $ f(z) $ satisfies all the conditions of {Theorem \ref{th2.1}} and \beas
		\Theta(0;f)+\Theta(\infty;f)=\frac{2}{3}=\frac{1}{k_1+1}+\frac{1}{k_2+1},\eeas
		\par where $ \Theta(0,f)={1}/{3}=\Theta(\infty,f) $, but $
		\mathcal{L}_{\pi i}(f)\not\equiv f $. 
	\end{exm}
	\begin{exm}
		Let $ f(z)={\left(4ae^z(e^{2z}+1)\right)}/{\left(e^{4z}+6e^{2z}+1\right)} $, $ c=\pi i
		$, $ s_1=a $, $ s_2=-a $, where $ a\in\mathbb{C^{*}} $, $ s_3=0 $, $ s_4=\infty
		$ and $ k_1=3=k_2 $. Then clearly $ \mathcal{L}_{\pi
			i}(f)=-{\left(4ae^z(e^{2z}+1)\right)}/{\left(e^{4z}+6e^{2z}+1\right)} $, where $
		c_1=c_0+1 $, $ c_0, c_1\in\mathbb{C^{*}} $, and $ f(z) $ satisfies all the
		conditions of \emph{Theorem \ref{th2.1}} and \beas
		\Theta(0;f)+\Theta(\infty;f)=\frac{1}{2}=\frac{1}{k_1+1}+\frac{1}{k_2+1},\eeas
		\par where  $ \Theta(0,f)={1}/{2}$, $\Theta(\infty,f)=0 $ but we see that $ \mathcal{L}_{\pi i}(f)\not\equiv f $. 
	\end{exm}
	\par As the consequences of {Theorem \ref{th2.1}}, we prove the following result.
	\begin{theo}\label{th2.2}
		Let $ f $ be a non-constant meromorphic function of hyper-order $ \rho_2(f)<1
		$, $ \Theta(\infty,f)=1 $ and $ c, c_1\in\mathbb{C^{*}} $. Let $ s_1, s_2,
		s_3\in\hat{\mathcal{S}}(f) $ be three distinct periodic functions with period $
		c $ such that $ f $ and $ \mathcal{L}_c(f) $ share $ s_3 $ $ CM $ and \beas \ol
		E_{k_{j})}(s_j,f)\subseteq \ol E_{k_{j})}(s_j,\mathcal{L}_c(f)),\;\; j=1, 2. 
		\eeas If $ k_1, k_2\geq 2 $, then $ \mathcal{L}_c(f)\equiv f $. Furthermore, $ f $
		assumes the following form \beas
		f(z)=\left(\frac{1-c_0}{c_1}\right)^{\displaystyle\frac{z}{c}}g(z) ,\eeas where
		$ g(z) $ is a meromorphic function such that $ g(z+c)=g(z) $, for all $
		z\in\mathbb{C} $.
	\end{theo} \par The following example shows that, the number $ k_1=2=k_2 $ is sharp in {Theorem \ref{th2.2}}.
	\begin{exm}\label{ex2.1}
	We consider $ f(z)=a\cos z $, where $ a\in\mathbb{C^{*}} $, $s_1=a, s_2=-a$ and
	$ s_3=0 $. We choose $ \mathcal{L}_{\pi}(f)=c_1f(z+\pi)+c_0f(z) $, where $
	c_1,\; c_0\in\mathbb{C^{*}} $ with $ c_1=c_0+1 $. Clearly $ f $ and $
		\mathcal{L}_{\pi}(f) $ share $ s_3 $ $ CM $, $ \Theta(\infty,f)=1 $, $ \ol
		E_{1)}(a,f)=\phi=\ol E_{1)}(a,\mathcal{L}_{\pi}(f)) $ and $ \ol
		E_{1)}(-a,f)=\phi=\ol E_{1)}(-a,\mathcal{L}_{\pi}(f)) $, but $ f(z)\not\equiv
		\mathcal{L}_{\pi}(f).$  
	\end{exm}
	\par Naturally, we are interested to find what happens, when $ k_1=1=k_2 $, and hence we obtain the following result.
	\begin{theo}\label{th2.3}
		Let $ f $ be a non-constant meromorphic function of hyper-order $ \rho_2(f)<1
		$ with $ \Theta(\infty,f)=1 $ and $ c, c_1\in\mathbb{C^{*}} $. Let $ s_1, s_2,
		s_3\in\hat{\mathcal{S}}(f) $ be three distinct periodic functions with period $
		c $ such that $ f $ and $ \mathcal{L}_c(f) $ share $ s_3 $ $ CM $ and \beas \ol
		E_{1)}(s_j,f)\subseteq \ol E_{1)}(s_j,\mathcal{L}_c(f)),\;\; j=1, 2.  \eeas Then
		$ \mathcal{L}_c(f)\equiv f $ or $ \mathcal{L}_c(f)\equiv -f+2s_3 $.Furthermore,
		\begin{enumerate}
			\item[(i)] If $ \mathcal{L}_c(f)\equiv f $, then 
			\beas f(z)=\left(\frac{
				1-c_0}{c_1}\right)^{\displaystyle\frac{z}{c}}g(z).\eeas
			\item[(ii)] If $ \mathcal{L}_c(f)\equiv -f+2s_3 $, then \beas
			f(z)=\left(\frac{-1-c_0}{c_1}\right)^{\displaystyle\frac{z}{c}}g(z)+2s_3,\;\;\text{for
				all}\; z\in\mathbb{C},\eeas
		\end{enumerate}
		where $ g(z) $ is a meromorphic function such that $ g(z+c)=g(z) $ Moreover, $
		\mathcal{L}_c(f)\equiv -f+2s_3 $ occurs only if $ s_1+s_2=2s_3 $.
	\end{theo} 
	\begin{rem}
		We see that {Theorems \ref{th2.1}, \ref{th2.2}}\; and {\ref{th2.3}}\;
		directly improved, respectively, {Theorems E, F} and {G}.
	\end{rem}
	\begin{rem}
		We see from {Example \ref{ex2.1}}\; that, in {Theorem
			\ref{th2.3}},\; the possibility $ \mathcal{L}_c(f)\equiv -f+2s_3 $ could be	occurred.
	\end{rem}\par The following example shows that the restrictions on the growth of
	$ f $ in our above results are necessary and sharp.
	\begin{exm}
		Let $ f(z)=e^{g(z)} $, where $ g(z) $ is an entire function with $ \rho(g)=1 $, and hence for $ c_1={1}/{2}=c_0 $,  it is easy to see that $
		\mathcal{L}_{\pi}(z)={\left(e^{2g(z)}+1\right)}/{2e^{g(z)}} $. We
		choose $ s_1=1,\; s_2=-1 $ and $ s_3=\infty $. Clearly $ \Theta(\infty,f)=1 $, $
		\rho_2(f)=1 $,  $ f $ and $ \mathcal{L}_{\pi}(f) $ share $ s_3 $ $ CM $ and $ \ol
		E_{1)}(1,f)\subseteq\ol E_{1)}(1,\mathcal{L}_{\pi}(z)) $ and $ \ol
		E_{1)}(-1,f)\subseteq\ol E_{1)}(-1,\mathcal{L}_{\pi}(z)) $ but we see that
		neither $ \mathcal{L}_{\pi}(z)\not\equiv f $ nor $
		\mathcal{L}_{\pi}(z)\not\equiv -f+2s_3 $. Also the function has not the specific
		form. 
	\end{exm} \par 
	\begin{rem}
		The next example shows that the condition $ \Theta(\infty,f)=1 $ in {Theorem \ref{th2.3}} can not be omitted. 
	\end{rem}
	\begin{exm}
		Let $ f(z)=1/ \cos z $, $c_1=1$, $ c_0=0 $, $ s_1=1 $, $ s_2=-1 $ and $ s_3=0
		$. Clearly $ \Theta(\infty,f)=0 $, $ f $ and $ \mathcal{L}_{3\pi/2}(f) $ share $
		s_3 $ $ CM $, $ \ol E_{1)}(1,f)\subseteq\ol E_{1)}(1,\mathcal{L}_{3\pi/2}(z)) $
		and $ \ol E_{1)}(-1,f)\subseteq\ol E_{1)}(-1,\mathcal{L}_{3\pi/2}(z)) $.
		However, one may observe that neither $ \mathcal{L}_{3\pi/2}(z)\not\equiv f $ nor $
		\mathcal{L}_{3\pi/2}(z)\not\equiv -f+2s_3 $. Also the function has not the specific
		form.
	\end{exm} 
	
	%-----------------------------------------------------------------------%
	
	\section{Key lemmas} 
	In this section, we present some necessary lemmas which will play key role to prove the main results. For a non-zero complex number $ c $ and for integers $ n\geq 1 $, we define the higher order difference operators $ \Delta_c^nf:=\Delta_c^{n-1}(\Delta_c f) $.
	\begin{lem}\cite{Yan-FE-1980}\label{lem3.1}
		Let $ c\in\mathbb{C} $, $ n\in\mathbb{N} $, let $ f $ be a meromorphic function
		of finite order. Then any small periodic function $ a\equiv
		a(z)\in\mathcal{S}(f) $ \beas 
		m\left(r,\frac{\Delta_c^nf}{f(z)-a(z)}\right)=S(r,f), \eeas where the	exponential set associated with $ S(r,f) $ is of at most finite logarithmic
		measure.
	\end{lem}
	\begin{lem}\cite{Moh-FFA-1971,Val-BSM-59}\label{lem3.2}
		If $ \mathcal{R}(f) $ is rational in $ f $ and has small meromorphic
		coefficients, then \beas T(r,\mathcal{R}(f))=\deg_f(\mathcal{R})T(r,f)+S(r,f).
		\eeas
	\end{lem}
	\begin{lem}\cite{Yan & Yi-2003}\label{lem3.3}
		Suppose that $ h $ is a non-constant entire function such that $ f(z)=e^{h(z)}
		$, then $ \rho_2(f)=\rho(h) $.
	\end{lem}
	\par In \cite{Chi & Fen-2008,Hal & Kor-JMMA-2006}, the first difference analogue
	of the lemma on the logarithmic derivative was proved and for the hyper-order $
	\rho_2(f)<1 $, the following is the extension, see \cite{Hal & Kor &
		Toh-TAMS-2014}. 
	\begin{lem}\cite{Hal & Kor & Toh-TAMS-2014}\label{lem3.4}
		Let $ f $ be a non-constant finite order meromorphic function and $
		c\in\mathbb{C} $. If $ c $ is of finite order, then \beas
		m\left(r,\frac{f(z+c)}{f(z)}\right)=O\left(\frac{\log r}{r} T(r,f)\right) \eeas
		for all $ r $ outside of a set $ E $ with zero logarithmic density. If the hyper
		order $ \rho_2(f)<1 $, then for each $ \epsilon>0 $, we have \beas
		m\left(r,\frac{f(z+c)}{f(z)}\right)=0\left(\frac{T(r,f)}{r^{1-\rho_2-\epsilon}}\right)
		\eeas for all $ r $ outside of a set of finite logarithmic measure.
	\end{lem}
	\begin{lem}\cite{Yamanoi-AM-2004}\label{lem3.5}
		Let $ f $ be a non-constant meromorphic function, $ s_j\in\hat{\mathcal{S}}(f)
		$, $ j=1, 2, ..., q,\;$ $ (q\geq 3) $. Then for any positive real number $
		\epsilon $, we have \beas (q-2-\epsilon)T(r,f)\leq\sum_{j=1}^{q}\ol
		N\left(r,\frac{1}{f-s_j}\right),\; r\not\in E, \eeas where $ E\subset [0,\infty)
		$ and satisfies $\displaystyle \int_{E}d\log \log r<\infty $.
	\end{lem}
We now prove the following lemma, a similar proof of this lemma can also be found in \cite{Aha & RM & 2019}.
	\begin{lem}\label{lem3.6}
		Let $ f $ be a non-constant meromorphic function such that \beas \ol
		E(s_j,f)\subseteq \ol E(s_j,c_1f(z+c)+c_0f(z)),\;\; j=1, 2,\eeas where $ s_1,
		s_2\in\mathcal{S}(f) $, $ c,\; c_0,\; c_1(\neq 0)\in\mathbb{C^{*}} $, then $ f $
		is not a rational. 
	\end{lem} 
	\begin{proof}
We wish to prove this lemma by the method of contradiction. Let $ f $ be a rational	function. Then $ f(z)={P(z)}/{Q(z)} $ where $ P $ and $ Q
$ are two polynomials relatively prime to each other and $ P(z)Q(z)\not\equiv 0 $. Hence \bea\label{e1.1} E(0,P)\cap E(0,Q)=\phi \eea It is easy to see that\beas
		c_1f(z+c)+c_0f(z)&=&c_1\frac{P(z+c)}{Q(z+c)}+c_0\frac{P(z)}{Q(z)}\\&=&\frac{c_1P(z+c)Q(z)+c_0P(z)Q(z+c)}{Q(z+c)Q(z)}\\&=&\frac{P_1(z)}{Q_{1}(z)},\text{(say)} \eeas where $ P_1 $ and $ Q_1 $ are two relatively prime polynomials and $ P_1(z)Q_1(z)\not\equiv 0 $.\vspace{1mm} 
		\par Since $ \ol
		E(s_1,f)\subseteq\ol E(s_1,c_1f(z+c)+c_0f(z)) $ and $ f $ is a rational function, there  must exists a polynomial $ h(z) $ such that \beas
		c_1f(z+c)+c_0f(z)-s_1=(f-s_1)h(z), \eeas which can be re-written as \bea\label{e2.1}
		\frac{c_1P(z+c)Q(z)+c_0P(z)Q(z+c)}{Q(z+c)Q(z)}-s_1\equiv
		\left(\frac{P(z)}{Q(z)}-s_1\right)h(z).\eea
		We now discuss the following cases:\\
		\noindent{\bf{Case 1.}} Let $ P(z) $ is non-constant.\par  Then by the
		\textit{Fundamental Theorem of Algebra}, there exists $ z_0\in\mathbb{C} $ such
		that $ P(z_0)=0 $. Then it follows from (\ref{e2.1}) that \bea\label{e2.3}
		c_1\frac{P(z_0+c)}{Q(z_0+c)}\equiv(1-h(z_0))s_1^{0},  \eea where $
		s_1^{0}=s_1(z_0) $.\par 
		\noindent{\bf{Subcase 1.1.}} Let $ z_0\in\mathbb{C} $ be such that $ s_1(z_0)=0
		$.\par Then from (\ref{e2.3}), it is easy to see that $ P(z_0+c)=0 $. Then we can deduce
		from (\ref{e1.1}) that $ P(z_0+mc)=0 $ for all positive integer $ m $. However,
		this is impossible, and hence we conclude that the polynomial $ P(z) $ is a
		non-zero constant.\par
		\noindent{\bf{Subcase 1.2.}} Let $ z_0\in\mathbb{C} $ be such that $
		s_1(z_0)\neq 0 $. \par Then from (\ref{e2.3}), we obtain that \beas 
		P(z_0+c)\equiv\frac{s_1^{0}}{c_1}(1-h(z_0))Q(z_0+c).\eeas\par Next proceeding
		exactly same way as done in above, we obtain \bea\label{e2.4} 
		P(z_0+mc)\equiv\frac{s_1^{0}}{c_1}(1-h(z_0))Q(z_0+mc).\eea\par In view of (\ref{e2.3})
		and (\ref{e2.4}), a simple computation shows that \beas
		\frac{P(z_0+c)}{Q(z_0+c)}=\frac{P(z_0+mc)}{Q(z_0+mc)}\;\; \text{for all
			positive integers $ m $,} \eeas which contradicts the fact that $ E(0,P)\cap
		E(0,Q)=\phi $.\par 
		Therefore, we see that $ f(z) $ takes the form $
		f(z)={\eta}/{Q(z)} $, where $ P(z)=\eta=\text{constant}\; (\neq
		0).$\par 
		\noindent{\bf{Case 2.}} Let $ Q(z) $ be non-zero constant.\par Now
		\bea\label{e2.5} c_1f(z+c)+c_0f(z)&=&\frac{c_1\eta\;Q(z)+c_0\eta\;
			Q(z+c)}{Q(z+c)Q(z)}.\eea\par Since $ E(s_2,f)=E(s_2,c_1f(z+c)+c_0f(z)) $ then
		there exists a polynomial $ h_1(z) $ such that $
		c_1f(z+c)+c_0f(z)-s_2=(f-s_2)h_1(z),$  which can be written as \bea\label{e2.6} c_1\;Q(z)+c_0
		Q(z+c)\equiv \frac{\eta-s_2Q(z)}{d}h_1(z)Q(z+c).\eea\par Since $ Q(z) $ and hence $ Q(z+c) $ be non-constant polynomials, hence by the \textit{Fundamental Theorem of Algebra}, there exist $ z_0 $ and $ z_1 $ such that $ Q(z_0)=0=Q(z_1+c) $.
		\par \noindent{\bf{Subcase 2.1.}} When $ Q(z_0)=0 $, then from (\ref{e2.6}), we
		see that $ h_1(z_0)=-{c_0}/{\eta} $, which is absurd.\par 
		\par \noindent{\bf{Subcase 2.2.}} When $ Q(z_1+c)=0 $, then from (\ref{e2.6}),
		we get $ Q(z_1)=0 $, which is not possible.\par 
		Thus we conclude that $ Q(z) $ is a non-zero constant, say $ \eta_2 $. Thus we have
		$ f(z)={\eta}/{\eta_2} $, a constant, which is a contradiction. This completes the proof.
	\end{proof}
	\begin{lem}\cite{Hal & Kor & Toh-TAMS-2014}\label{lem3.7}
		Let $ T: [0,+\infty]\rightarrow [0,+\infty] $ be a non-decreasing continuous
		function, and let $ s\in (0,+\infty) $. If the hyper-order of $ T $ is strictly
		less than one, i.e., \beas \limsup_{r\rightarrow +\infty}\frac{\log ^{+}\log
			^{+} T(r)}{\log r}=\rho_2 <1,\eeas then \beas
		T(r+s)=T(r)+o\left(\frac{T(r)}{r^{1-\rho_2-\epsilon}}\right),\eeas where $
		\epsilon>0 $ and $ r\rightarrow\infty $, outside of a set of finite logarithmic
		measure.
	\end{lem}
	%------------------------------------------------------------------------%
	\section{Proofs of the main results}
	In this section, we give the proofs of our main results.
	\begin{proof}[Proof of Theorem \ref{th2.1}] 	
		First of all we suppose that $ s_j\in\mathbb{C} $, $ j=1, 2, 3, 4 $. By the
		assumption of the theorem, $ f(z) $ and $ \mathcal{L}_c(f)=c_1f(z+c)+c_0f(z) $
		share $ s_3 $, $ s_4 $ $ CM $, hence we must have \bea\label{e3.1}
		\frac{\left(f-s_3\right)\left(\mathcal{L}_c(f)-s_4\right)}{\left(f-s_4\right)\left(\mathcal{L}_c(f)-s_3\right)}=e^{h(z)},
		\eea where $ h(z) $ is an entire function with $ \rho(h)<1 $ by Lemma \ref{lem3.3}. In view of Lemma \ref{lem3.4}, we obtain \beas
		T\left(r,e^h\right)=S(r,f).  \eeas\par Therefore with the help of Lemma
		\ref{lem3.2}, we obtained \beas T(r,\mathcal{L}_c(f))=T(r,f)+S(r,f). \eeas
		\par Next we suppose that $ z_0\in\ol E_{k)}(s_1,f)\cup\ol E_{k)}(s_2,f) $. Then from	(\ref{e3.1}), one may easily deduce that $ e^{h(z_0)}=1 $. For the sake of
		convenience, we set $ \gamma :=e^{h(z)} $ and \beas  S(r)
		:=S(r,\mathcal{L}(f))=S(r,f).\eeas \par We now split the
		problem into two cases.\par 
		\noindent{\bf Case 1.} Let $ e^{h(z)}\neq 1 $.\par  A simple computation shows that
		that \bea\label{e3.2} \ol N_{k_1)} \left(r,\frac{1}{f-s_1}\right)&\leq&
		N\left(r,\frac{1}{\gamma -1}\right)\leq T(r,\gamma)+O(1)\leq S(r) \eea and
		\bea\label{e3.3} \ol N_{k_2)} \left(r,\frac{1}{f-s_2}\right)&\leq&
		N\left(r,\frac{1}{\gamma -1}\right)\leq T(r,\gamma)+O(1)\leq S(r). \eea \par
		Without loss of generality, we may assume that $ s_3 $, $ s_4
		\in\mathcal{S}(f)\setminus\{0\}$. By Lemma \ref{lem3.5}, for \beas
		\epsilon\in\left(0,\frac{1}{3}\left(\Theta(0;f)+\Theta(\infty;f)\right)-\frac{1}{k_1+1}-\frac{1}{k_2+1}\right),
		\eeas we obtain \bea\label{e3.4} (4-\epsilon) T(r,f)\leq \ol N(r,f)+\ol
		N\left(r,\frac{1}{f}\right)+\sum_{j=1}^{4}\ol
		N\left(r,\frac{1}{f-s_j}\right)+S(r,f). \eea\par With the help of (\ref{e3.2})
		and (\ref{e3.3}), it follows from (\ref{e3.4}) that \beas (2-\epsilon) T(r,f)\leq
		\ol N(r,f)+\ol N\left(r,\frac{1}{f}\right)+\sum_{j=1}^{2}\ol
		N_{(k_j+1}\left(r,\frac{1}{f-s_j}\right)+S(r,f) \eeas which gives \beas
		\Theta(0;f)+\Theta(\infty;f)\leq\frac{1}{k_1+1}+\frac{1}{k_2+1} \eeas which
		contradicts \beas \Theta(0;f)+\Theta(\infty;f)>\frac{1}{k_1+1}+\frac{1}{k_2+1}.
		\eeas
		\noindent{\bf Case 2.} Therefore, we have $ e^{h(z)}\equiv 1 $ and hence \beas
		\frac{(f-s_3)(\mathcal{L}_c(f)-s_4)}{(f-s_4)(\mathcal{L}_c(f)-s_3)}=1, \eeas
		on simplification, we obtain $ \mathcal{L}_c(f)\equiv f(z) $, for all $
		z\in\mathbb{C} $.\par We are now to find the class of all the
		meromorphic functions satisfying the difference equation $ \mathcal{L}_c(f)\equiv f $. By assumption of the result, and using {Lemma \ref{lem3.6}}, it is easy to see that $ f $
		is not a rational function. Therefore $ f(z) $ must be a transcendental	meromorphic function. \par We also see that $ f(z) $ and $ f(z+c) $ are related by \bea\label{e3.5}
		f(z+c)=\left(\frac{1-c_0}{c_1}\right)f(z). \eea \par Let $ f_1(z) $ and $ f_2(z)
		$ be two solutions of (\ref{e3.5}) (see \cite{Aha & RM & 2019} for more details). Then it is easy to see that\bea\label{e3.6}
		f_1(z+c)=\left(\frac{1-c_0}{c_1}\right)f_1(z)\\\label{e3.7}
		f_2(z+c)=\left(\frac{1-c_0}{c_1}\right)f_2(z). \eea \par We set $
		h(z):=f_1(z)/f_2(z) $. Then in view of (\ref{e3.6}) and (\ref{e3.7}), we
		obtain \beas
		h(z+c)=\frac{f_1(z+c)}{f_2(z+c)}=\displaystyle\frac{\displaystyle\frac{1-c_0}{c_1}f_1(z)}{\displaystyle\frac{1-c_0}{c_1}f_2(z)}=\frac{f_1(z)}{f_2(z)}=h(z),\eeas
		for all $ z\in\mathbb{C} $. Therefore, it is easy to verify that $$f_2(z)=\displaystyle\left(\frac{1-c_0}{c_1}\right)^{\displaystyle\frac{z}{c}}g_2(z),$$ where $ g_2(z) $ is a meromorphic function with $ g_2(z+c)=g_2(z) $, is a	solution of (\ref{e3.5}). Hence, it is also easy to verify that $ f_1(z)=f_2(z)h(z) $, a solution of (\ref{e3.5}). Thus the linear combination  \beas a_1f_1(z)+a_2f_2(z)&=&
		\displaystyle\left(\frac{1-c_0}{c_1}\right)^{\displaystyle\frac{z}{c}}\left(a_1h(z)+a_2\right)g_2(z)\\&=&\displaystyle\left(\frac{1-c_0}{c_1}\right)^{\displaystyle\frac{z}{c}}\sigma
		(z),\eeas where $ \sigma(z)=\left(a_1h(z)+a_2\right)g_2(z) $ is such that $
		\sigma(z+c)=\sigma(z) $, for all $ z\in\mathbb{C} $, is the general solution of
		(\ref{e3.5}). Hence, the precise form of $ f(z) $ is \beas 
		f(z)=\displaystyle\left(\frac{1-c_0}{c_1}\right)^{\displaystyle\frac{z}{c}}g(z),
		\eeas where $ g(z) $ is a meromorphic function with $ g(z+c)=g(z) $, for all $ z\in\mathbb{C} $.\vspace{1mm} \par This completes the proof.
	\end{proof}
	\begin{proof}[Proof of Theorem \ref{th2.3}] 
		Let us suppose that $ g(z) $ is the canonical product of the poles of $ f $. Then by Lemma \ref{lem3.4}, we obtained \bea\label{e4.8}
		m\left(r,\frac{g(z+c)}{g(z)}\right)=S(r,f)=S(r,f). \eea\par Since $ \Theta(\infty;f)=1 $, hence it is easy to see that \beas \limsup_{r\rightarrow +\infty}\frac{\ol
			N(r,f)}{T(r,f)}=0. \eeas\par Therefore it follows from (\ref{e4.8}) that\bea\label{e4.9} T\left(r,\frac{g(z+c)}{g(z)}\right)=S(r,f). \eea\par Since $ f
		$ and $ \mathcal{L}_c(f) $ share $ s_3 $ $ CM $, by
		Lemma \ref{lem3.3}, we obtain
		\bea\label{e4.10}
		\frac{\mathcal{L}_c(f)-s_3}{f-s_3}=e^{\mathcal{H}(z)}\frac{g(z)}{g(z+c)}, \eea
		where $ \mathcal{H}(z) $ is an entire function, with $ \rho(\mathcal{H})<1 $. By Lemma \ref{lem3.4}, we also obtain
		\bea\label{e4.11} T\left(r,e^{\mathcal{H}(z)}\frac{g(z)}{g(z+c)}\right)=S(r,f).
		\eea\par Therefore, by Lemma \ref{lem3.2} and (\ref{e4.11}), a simple computation shows that $
		T(r,\mathcal{L}_c(f))=T(r,f)+S(r,f). $ For the sake convenience, we set  \beas
		\beta:=e^{\mathcal{H}(z)}\frac{g(z)}{g(z+c)}\;\;\text{and}\;\;
		S(r):=S(r,\mathcal{L}_c(f))=S(r,f). \eeas\par If $ \mathcal{L}_c(f)\not\equiv
		f(z) $. i.e., if $ \beta\neq 1 $, then with the help of (\ref{e4.10}) and from the assumption, we obtain \bea\label{e4.12} \ol
		N_{1)}\left(r,\frac{1}{f-s_1}\right)&\leq& N\left(r,\frac{1}{\beta
			-1}\right)\leq T(r,\beta)+O(1)= S(r). \eea and \bea\label{e4.13} \ol
		N_{1)}\left(r,\frac{1}{f-s_2}\right)&\leq& N\left(r,\frac{1}{\beta
			-1}\right)\leq T(r,\beta)+O(1)= S(r). \eea By Lemma \ref{lem3.7}, and using
		(\ref{e4.12}) and (\ref{e4.13}), we easily obtain \bea\label{e4.14} \ol
		N_{1)}\left(r,\frac{1}{\mathcal{L}_c(f)-s_1}\right)\leq \ol
		N_{1)}\left(r,\frac{1}{f-s_1}\right)+S(r)=S(r). \eea and \bea\label{e4.15} \ol
		N_{1)}\left(r,\frac{1}{\mathcal{L}_c(f)-s_2}\right)\leq \ol
		N_{1)}\left(r,\frac{1}{f-s_1}\right)+S(r)=S(r). \eea\par On the other hand, it follows from
		(\ref{e4.10}) that \bea\label{e4.16} \mathcal{L}_c(f)-s_1&=&(s_3-s_1)+\beta\;
		(f-s_3)\\&=&\beta\;\left(f-\frac{s_1+(\beta -1)s_3}{\beta}\right)\nonumber. \eea
		\par Similarly, we obtain \bea\label{e4.17}
		\mathcal{L}_c(f)-s_2=\beta\;\left(f-\frac{s_2+(\beta
			-1)s_3}{\beta}\right).\eea\par It is easy to see that \bea\label{e4.18}
		N\left(r,\frac{1}{\mathcal{L}_c(f)-s_1}\right)=N\left(r,\frac{1}{f-\frac{s_1+(\beta
				-1)s_3}{\beta}}\right)+S(r). \eea and \bea\label{e4.19}
		N\left(r,\frac{1}{\mathcal{L}_c(f)-s_2}\right)=N\left(r,\frac{1}{f-\frac{s_2+(\beta
				-1)s_3}{\beta}}\right)+S(r). \eea\par Now our aim is to deal with the following three cases.\vphantom{1mm}\\ 
	\noindent{\bfseries{Case 1.}} Suppose that $ \left({\left((\beta-1)s_3+s_1\right)}/{\beta}\right)\neq s_2 $.\par  Since $ \left({\left((\beta
			-1)s_3+s_1\right)}/{\beta}\right)\neq s_1 $ and $ \Theta(\infty;f)=1 $, hence by Lemma
		\ref{lem3.5} for $ \epsilon\in \left(0,{1}/{2}\right) $, it
		follows from (\ref{e4.10}), (\ref{e4.12}), (\ref{e4.13}), (\ref{e4.14}) and
		(\ref{e4.18}) that \bea  && (2-\epsilon) T(r,f)\\&\leq& \ol N(r,f)+\ol
		N\left(r,\frac{1}{f}\right)+\ol N\left(r,\frac{1}{f-s_2}\right)+\ol
		N\left(r,\frac{1}{f-\frac{(\beta -1)s_3+s_1}{\beta}}\right)\nonumber\\&\leq& \ol
		N_{(2}\left(r,\frac{1}{f-s_1}\right)+\ol
		N_{(2}\left(r,\frac{1}{f-s_2}\right)+\ol
		N_{(2}\left(r,\frac{1}{\mathcal{L}_c(f)-s_1}\right)\nonumber\\&\leq&\frac{1}{2}
		T(r,f)+\frac{1}{2} T(r,f)+\frac{1}{2} T(r,f)+S(r)\nonumber\\&=&\frac{3}{2}
		T(r,f)+S(r,f)\nonumber, \eea which is a contradiction.\vspace{1mm} \\
		\noindent{\bfseries{Case 2.}} Suppose that $ \left({\left((\beta
			-1)s_3+s_2\right)}/{\beta}\right)\neq s_1 $.\par  Since $ \left({\left((\beta
			-1)s_3+s_2\right)}/{\beta}\right)\neq s_2 $ and $ \Theta(\infty;f)=1 $, hence proceeding
		exactly same way as done {Case 1}, we arrive at a contradiction.\par Therefore,
		we must have $ \mathcal{L}_c(f)\equiv f $, and hence following the proof of {Theorem \ref{th2.1}}, we obtain the precise form of the function. \vspace{1mm}\\ 
\noindent{\bfseries{Case 3.}} Suppose that \beas \frac{(\beta
			-1)s_3+s_2}{\beta}=s_1 \eeas and \beas \frac{(\beta -1)s_3+s_1}{\beta}=s_2
		.\eeas\par An elementary calculation shows that $ \beta =-1 $, so that $	2s_3=s_1+s_2 $. Thus from (\ref{e4.10}), we have $ \mathcal{L}_c(f)\equiv -f(z)+2s_3 $ and by the same argument used in the previous cases, it is not hard to show that $ f(z) $ will take the form \beas
		f(z)=\left(\frac{-1-c_0}{c_1}\right)^{\displaystyle\frac{z}{c}}g(z)+2s_3,\;\;\text{for all}\; z\in\mathbb{C},  \eeas where $ g(z) $ is a meromorphic function with period $ c $. This completes the proof. 
	\end{proof}
	
	%________________________________________________________________________________________________________%
	
	\section{Concluding remarks and open question}
	\par Let us suppose that $ \mathcal{L}_c(f)\equiv f $, where $ f $ is a non-constant meromorphic functions. Since $ f $ can not be rational function (see \cite{Aha & RM & 2019} for detailed information), hence $ f $ must be transcendental and hence $ f(z) $ takes the precise form \beas 
	f(z)=\left(\frac{1-c_0}{c_1}\right)^{\displaystyle\frac{z}{c}}g(z), \eeas where $ g(z) $ is a meromorphic periodic function $ c $. We can write $ f(z)=\alpha^{\frac{z}{c}}g(z), $ where $ \alpha $ is a root of the equation $	c_1z+c_0=1 $.\vspace{1mm} 
	
	\par For further generalization, we define $
	\mathcal{L}_c^n(f):=c_nf(z+nc)+\cdots+c_1f(z+c)+c_0f(z) $ (see \cite{Ban & Aha & JCMA-2020} for details), where $ c_n(\neq 0), c_1, c_0\in\mathbb{C} $. For particular values of the constants $ c_j=(-1)^{n-j}\binom nj $ for $ j=0, 1, \ldots, n $, it is easy to see that $\mathcal{L}_c^n(f)=\Delta_{c}^n(f). $	One can check that $ f(z)=2^{^{\frac{z}{c}}}g(z) $, where $ g $ is a
	meromorphic function of period $ c $, solves the difference equation $ \Delta_{c}^n(f)\equiv f $. 	We are mainly interested to find the precise form of the function $ f $ when  it solves the difference equation $ \mathcal{L}_c^n(f)\equiv f $. However, we conjecture the following.\\
	
	\noindent{\bf Conjecture:} 	Let $ f $ be a meromorphic function such that $ \mathcal{L}_c^n(f)\equiv f $,
	then $ f $ assumes the form $$ f(z)=\alpha_n^{{z}/{c}}g_n(z)+\cdots+\alpha_1^{{z}/{c}}g_1(z), $$
	where $ g_j $ $ (j=1, 2, \ldots, n) $ are meromorphic functions of period $ c $, and $ \alpha_j $ $ (j=1, 2, \ldots, n) $ are roots of the equation $ c_nz^n+\cdots+c_1z+c_0=1 $. \vspace{1mm}

	\par Based on the above discussions, we now pose the following question for future investigations on the main results of the paper.
\begin{ques}
Keeping all other conditions intact, for a meromorphic function $ f $, is it possible to get a corresponding result of {Theorems \ref{th2.1},
\ref{th2.2}} and {\ref{th2.3}}\; for $ \mathcal{L}_c^n(f) $?
\end{ques}

\end{document}